\documentclass[11pt]{article}
\usepackage{amsthm,amsfonts,amsmath,amstext,amssymb,amsopn}
\usepackage{graphicx}
\usepackage{xifthen}
\usepackage{multirow}
\usepackage{caption}
\usepackage{booktabs}
\usepackage{floatrow}
\usepackage{color}
\usepackage{comment}
\usepackage{enumitem}
\usepackage{algorithm}
\usepackage{algpseudocode}
\usepackage{rotating}
\usepackage{chngpage}
\usepackage{floatrow}
\usepackage{url}

\begin{document}
\newtheorem{myclaim}{Claim}
\newtheorem{observation}{Observation}
\newtheorem{proposition}{Proposition}
\newtheorem{definition}{Definition}
\newtheorem{construction}{Construction}
\newtheorem{conjecture}{Conjecture}
\newtheorem{example}{Example}
\newtheorem{ques}{Question}
\newtheorem{lemma}{Lemma}
\newtheorem{corollary}[proposition]{Corollary}
\newtheorem{theorem}{Theorem}
\newcommand{\rr}{\mathbb{R}}
\newcommand{\zz}{\mathbb{Z}}
\newcommand{\bb}{\mathbb{B}}
\newcommand{\nn}{\mathbb{N}}
\newcommand{\mn}{\textup{min}}
\newcommand{\mx}{\textup{max}}
\newcommand{\st}{\textup{s.t.}}
\newcommand{\ignore}[1]{}

\def\bred{\color{red}}

\title{On a Cardinality-Constrained Transportation Problem With Market Choice}
\author{Pelin  Damc{\i}-Kurt\footnote{damci-kurt.1@osu.edu, Department of Integrated Systems Engineering, The Ohio State University, Columbus, OH 43210, United States} \and Santanu S. Dey\footnote{santanu.dey@isye.gatech.edu, School of Industrial and Systems Engineering, Georgia Institute of Technology, Atlanta, GA 30332, United States} \and Simge   K\"{u}\c{c}\"{u}kyavuz\footnote{kucukyavuz.2@osu.edu, Department of Integrated Systems Engineering, The Ohio State University, Columbus, OH 43210, United States }}
\date{}

\maketitle
%

\begin{abstract}
It is well-known that the intersection of the matching polytope with a cardinality constraint is integral~\cite{schrijver}. We prove a similar result for the polytope corresponding to the transportation problem with market choice (TPMC) (introduced in~\cite{DDK14}) when the demands are in the set $\{1,2\}$. This result generalizes the result regarding the matching polytope and also implies that some special classes of minimum weight perfect matching problem with a cardinality constraint on a subset of edges can be solved in polynomial time.
\end{abstract}


\section{Introduction and Main Result}\label{sec:intro}

\subsection{Transportation Problem with Market Choice}
The transportation problem with market choice (TPMC), introduced in the paper~\cite{DDK14}, is a transportation problem in which suppliers with limited capacities have a choice of which demands (markets) to satisfy. If a market is selected, then its demand must be satisfied fully through shipments from the suppliers. If a market is rejected, then the corresponding potential revenue is lost. The objective is to minimize the total cost of shipping and lost revenues.  See~\cite{geunes,levi,vanDenHeuvel} for approximation algorithms and heuristics  for several other  supply chain planning and logistics problems with market choice.

Formally, we are given a set of supply and demand nodes that form a bipartite graph $G(V_1 \cup V_2, E)$. The nodes in set $V_1$ represent the supply nodes, where for $i\in V_1$, $s_i\in \mathbb{N}$  represents the capacity of supplier $i$. The nodes in set $V_2$ represent the potential markets, where for $j\in V_2$, $d_j\in \mathbb{N}$ represents the demand of market $j$. The edges between supply and demand nodes have weights that represent shipping costs $w_{ij}$, where $(i,j) \in E$.
For each $j \in V_2$, $r_j$ is the revenue lost if the market $j$ is rejected. Let $x_{ij}$ be the amount of demand of market $j$ satisfied by supplier $i$ for $(i,j)\in E$, and let $z_{j}$ be an indicator variable taking a value 1 if market $j$ is rejected and 0 otherwise. A mixed-integer programming (MIP) formulation of the problem is given where the objective is to minimize the transportation costs and the lost revenues due to unchosen markets:
\begin{align}  \label{TPSL:objective}
\min_{x \in \rr^{|E|}_{+}, z \in \{0, 1\}^{|V_2|}} &\quad  \sum_{(i,j)\in E} w_{ij}x_{ij} + \sum_{j \in V_2} r_{j}z_{j} \\ \label{TPSL:demandcons}
\st & \quad \sum_{i: (i,j)\in E} x_{ij}  = d_j (1-z_j) & \forall j \in V_2\\
\label{TPSL:supplycons}  & \quad \sum_{j: (i,j)\in E} x_{ij} \leq s_i & \forall i \in V_1.
\end{align}
We refer to the formulation \eqref{TPSL:objective}-\eqref{TPSL:supplycons} as TPMC. The first set of constraints \eqref{TPSL:demandcons}  ensures that if market $j\in V_2$ is selected (i.e., $z_j=0$), then its demand must be fully satisfied.  The second set of constraints \eqref{TPSL:supplycons} model the supply restrictions.

TPMC is strongly NP-complete in general~\cite{DDK14}. The paper \cite{AB14} give polynomial-time reductions from this problem to the capacitated facility location problem~\cite{KH63}, thereby establishing approximation algorithms with constant factors for the metric case and a logarithmic factor for the general case.

\subsection{TMPC with $d_j \in \{1,2\}$ for all $j \in V_2$ and the Matching Polytope}
When $d_j \in \{1,2\}$ for each demand node $j \in V_2$, TPMC is polynomially solvable~\cite{DDK14}. This is proven through a reduction to a minimum weight perfect matching problem on a general (non-bipartite) graph $G'=(V',E')$; see~\cite{DDK14}. We call this special class of the problem, the {\it simple TPMC problem} in the rest of this note.

\begin{observation}[Simple TPMC generalizes Matching on General Graphs]\label{mathobs}
The matching problem can be seen as a special case of the simple TPMC problem. Let $G(V,E)$ be a graph with $n$ vertices and $m$ edges. We construct a bipartite graph {$\hat{G}(\hat{V}^1 \cup \hat{V}^2, \hat{E})$} as follows: {$\hat{V}^1$} is a set of $n$ vertices corresponding to the $n$ vertices in $G$, and {$\hat{V}^2$} corresponds to the set of edges of $G$, i.e., {$\hat{V}^2$} contains $m$ vertices. We use $(i,j)$ to refer to the vertex in {$\hat{V}^2$} corresponding to the edge $(i,j)$ in $E$. The set of edges in {$\hat{E}$} are of the form $(i, (i,j))$ and $(j, (i,j))$ for every $i, j \in V$ such that $(i,j) \in E$. Now we can construct (the feasible region of) an instance of TPMC with respect to {$\hat{G}(\hat{V}^1 \cup \hat{V}^2, \hat{E})$} as follows:
\begin{eqnarray}
\label{match1}Q= \{ (x, z) \in \mathbb{R}^{2m} \times \mathbb{R}^{m} \, |\, x_{i, (i,j)} + x_{j, (i,j)} {+ 2z_{(i,j)}} = 2 \ \forall (i,j) \in {\hat{V}^2} \\
\label{match2}\sum_{j:(i,j) \in E} x_{i, (i,j)} \leq 1 \ \forall i \in {\hat{V}^1} \\
\label{match3} z_{(i,j)} \in \{0,1\} \ \forall (i,j) \in {\hat{V}^2} \}.
\end{eqnarray}

Clearly there is a bijection between the set of matchings in $G(V,E)$ and the set of solutions in $Q$. Moreover, let $$H: = \{ (x,z,y) \in \mathbb{R}^{2m} \times \mathbb{R}^{m} \times \mathbb{R}^{m}\,|\, (x,z) \in Q, y = e - z\},$$ where $e$ is the all ones vector in $\mathbb{R}^m$. Then we have that the incidence vector of all the matchings in $G(V,E)$ is precisely the set $\textup{proj}_y(H)$.

Note that the instances of the form of (\ref{match1})-(\ref{match3}) are special cases of simple TPMC instances, since in these instances all $s_i$'s are restricted to be exactly $1$ and all $d_j$'s are restricted to be exactly $2$.
\end{observation}

\subsection{Simple TPMC with Cardinality Constraint: Main Result}

An important and natural constraint that one may add to the TPMC problem is that of a service level, that is the number of rejected markets is restricted to be at most $k$. This restriction can be modelled using a \emph{cardinality constraint}, $\sum_{j\in V_2} z_j \le k$, appended to \eqref{TPSL:objective}-\eqref{TPSL:supplycons}. We call the resulting problem cardinality-constrained TPMC (CCTPMC). If we are able to solve CCTPMC in polynomial-time, then we can solve TPMC in polynomial time by solving CCTPMC for all $k \in \{0, \dots, |V_2|\}$. Since TPMC is NP-hard,  CCTPMC is NP-hard in general.

In this note, we examine the effect of  appending a cardinality constraint  to the simple TPMC problem.

\begin{theorem}\label{conjectureintproperty} Given an instance of TPMC with $V_2$, the set of demand nodes, and $E$, the set of edges, let $X \in \mathbb{R}^{|E|}_{+} \times \{0,1\}^{|V_2|}$ be the set of feasible solutions of this instance of TPMC. Let $k \in \mathbb{Z}_{+}$ and $k \leq |V_2|$. Let $X^k := \textup{conv}(X \cap \{ (x,z) \in \rr^{|E|}_{+} \times \{0, 1\}^{|V_2|}\,| \sum_{j \in V_2} z_j \leq k\}).$ If $d_j\leq2$ for all $j \in V_2$, then $X^k = \textup{conv} (X) \cap \{ (x,z) \in \rr^{|E|}_{+} \times [0, 1]^{|V_2|}\,|\, \sum_{j \in V_2} z_j \leq k\}.$
\end{theorem}

Our proof of Theorem \ref{conjectureintproperty} is presented in Section 2. We note that the result of Theorem \ref{conjectureintproperty} holds even when $X^k$ is defined as $\textup{conv}(X \cap \{ (x,z) \in \rr^{|E|}_{+} \times \{0, 1\}^{|V_2|}\,| \sum_{j \in V_2} z_j \geq k\})$ or $\textup{conv}(X \cap \{ (x,z) \in \rr^{|E|}_{+} \times \{0, 1\}^{|V_2|}\,| \sum_{j \in V_2} z_j = k\})$.

By invoking the ellipsoid algorithm and the use of Theorem \ref{conjectureintproperty} we obtain the following corollary.
\begin{corollary}\label{corrolary1}
Cardinality constrained simple TPMC is polynomially solvable.
\end{corollary}

We note two other consequences of Theorem~\ref{conjectureintproperty}:
\begin{enumerate}
\item Special class of minimum weight perfect matching problem with a cardinality constraint on a subset of edges can be solved in polynomial time: As discussed in the previous section, simple TPMC can be reduced to a minimum weight perfect matching problem on a general (non-bipartite) graph $G'=(V',E')$~\cite{DDK14}. Therefore, it is possible to reduce CCTPMC with $d_j \leq 2$ for all $j\in V_2$ to a \emph{minimum weight perfect matching problem with a cardinality constraint on a subset of edges}.
    Hence, Corollary \ref{corrolary1} implies that a special class of minimum weight perfect matching problem with a cardinality constraint on a subset of edges can be solved in polynomial time.

    Note that the intersection of the perfect matching polytope with a cardinality constraint on a strict subset of edges is not always integral.

    \begin{example}\label{example:pefectmatch} Consider the bipartite graph $G(V_1\cup V_2,E)$ with $V_1 = \{1, 2, 3\}, V_2=\{ 4, 5, 6\}$, $E = \{(1,4), (1,5), $ $ (2,4), (2, 5), (2,6), (3,5), (3, 6)\}$, and the cardinality constraint $x_{14} + x_{25} \leq 1$. It is straightforward to show that  $x_{14} = x_{15} =x_{24} =x_{25} =0.5, x_{26} = x_{3 5} =0, x_{36} = 1$  is a fractional extreme point of the intersection of the perfect matching polytope with the cardinality constraint.
    \end{example}

    To the best of our knowledge, the complexity status of minimum weight perfect matching problem on a general graph with a cardinality constraint on a subset of edges is open. This can be seen by observing that if one can solve minimum weight perfect matching problem with a cardinality constraint on a subset of edges in polynomial time, then one can solve the exact perfect matching problem in polynomial time; see discussion in the last section in~\cite{bergerbogrsc2011}.

\item Generalization of the matching cardinality result: A well-known result is that the intersection of the matching polytope with a cardinality constraint is integral~\cite{schrijver}. It is straightforward to verify that this result follows from Theorem \ref{conjectureintproperty} applied to instances of TPMC constructed in Observation \ref{mathobs} (See Appendix A). However as mentioned in Observation \ref{mathobs}, the instances (\ref{match1})-(\ref{match3}) are special cases of simple TPMC instances. Therefore, Theorem~\ref{conjectureintproperty} represents a generalization of the classical result of integrality of the intersection of matching polytope and a cardinality constraint.
\end{enumerate}

Finally we ask the natural question: Does the statement of Theorem 1 hold when $d_j \leq 2$ does not hold for every $j$? The next example illustrates that the statement of Theorem 1 does not hold in such case.

\begin{example}\label{ex:CounterExample1}
Consider an instance of TPMC where $G(V_1 \cup V_2, E)$ is a bipartite graph with $V_1 = \{1,2,\ldots,6\}$, $V_2 = \{1,2,3,4\}$, $E=\{(1,1),  (2,2),  (3,3),  (4,1),   $ $ (4,4), (5,2), (5,4), (6,3), (6,4)\}$,  $s_i=1$, $i \in V_1$, $d_j=2$, $j = \{1,2,3\}$, $d_4=3$.   For  $k=2$ we it can be verified that we obtain a non-integer extreme point of $\textup{conv}(X) \cap \{(x,z) \in \rr^{p}_{+} \times [0,1]^n | \sum_{j=1}^{n} z_j \leq k \}$, given by $x_{11}=x_{22}=x_{33}=x_{41}=x_{44}=x_{52}=x_{54}=x_{63}=x_{64}=z_{1}=z_{2}=z_{3}=z_{4}=\frac{1}{2}$. Therefore, $X^k\ne \textup{conv}(X) \cap \{(x,z) \in \rr^{p}_{+} \times [0,1]^n | \sum_{j=1}^{n} z_j \leq k \}$ in this example.
\end{example}

\section{Proof of Theorem \ref{conjectureintproperty}}\label{sec:cardinality}

To prove Theorem \ref{conjectureintproperty}, one approach could be to appeal to the reduction to minimum weight perfect matching problem and then use the well-known adjacency properties of the vertices of the perfect matching polytope. However, as illustrated in Example \ref{example:pefectmatch}, the integrality result does not hold for the perfect matching polytope on a general graph with a cardinality constraint on any subset of edges. Therefore a generic approach considering the perfect matching polytope appears to be less fruitful. We use an alternative approach to prove this result. In particular, we apply a technique similar to that used in~\cite{aghezzaf}. Consider the following property:

\begin{definition}[Edge Property]\label{edgeproperty}
Let $T \subseteq \rr_{+}^{p} \times \{0,1\}^{n}$ be some mixed integer set. We say that $T$ satisfies the \emph{edge property} if for all $(w,r) \in \rr^{p+n}$ such that $\mn \{w^{\top}x + r^{\top}z\,|\, (x,z) \in T\}$ is bounded and has at least two optimal solutions, $(x^{1}, z^1)$ and $(x^2, z^2)$ where $\sum_{j=1}^{n} z^1_j = k^1$, $\sum_{j=1}^{n} z^2_j = k^2$ and $k^1 \leq k^2 - 2$, then there is an optimal solution $(x^3,z^3)$ such that $\sum_{j=1}^{n} z^3_j = k^3$ and $ k^1 < k^3 < k^2$.
\end{definition}

\begin{proposition}\label{edgproperty-convexhull}
Let $T \subseteq \rr_{+}^{p} \times \{0,1\}^{n}$ be a mixed integer set such that $\textup{conv}(T)$ is a pointed polyhedron and let $T^{k} := \textup{conv}(T \cap \{(x,z) \in \rr^{p}_{+} \times \{0,1\}^{n} \,|\, \sum_{j=1}^{n} z_j \leq k \}).$ If $T$ satisfies the edge property, then $T^{k} = \textup{conv}(T) \cap \{(x,z) \in \rr^{p}_{+} \times [0,1]^n | \sum_{j=1}^{n} z_j \leq k \}$.
\end{proposition}
We present the proof of the Proposition~\ref{edgproperty-convexhull} for completeness. See also~\cite{aghezzaf}.
\begin{proof}

Assume by contradiction that $$T^k \not =\textup{conv}(T) \cap \{(x,z) \in \rr^{p}_{+} \times [0,1]^n | \sum_{j=1}^{n} z_j \leq k \},$$ for some $k=k' \in \{0,1,\ldots,n\}$. By definition $T^{k} = \textup{conv}(T \cap \{(x,z) \in \rr^{p}_{+} \times \{0,1\}^{n} | \sum_{j=1}^{n} z_j \leq k \})$ so $T^k \subseteq \textup{conv}(T) \cap \{(x,z) \in \rr^{p}_{+} \times [0,1]^n | \sum_{j=1}^{n} z_j \leq k \}$ holds for all $k \in \{0,1,\ldots,n\}$. By assumption we obtain $T^{k'} \subset \textup{conv}(T) \cap \{(x,z) \in \rr^{p}_{+} \times [0,1]^n | \sum_{j=1}^{n} z_j \leq k' \}$. Since $\textup{conv}(T)$ is pointed this implies that there exists a vertex $(x',z')$ of $ \textup{conv}(T) \cap \{(x,z) \in \rr^{p}_{+} \times [0,1]^n | \sum_{j=1}^{n} z_j \leq k' \}$ such that $(x',z') \not \in T^{k'}$.
Therefore $z'$ is fractional and $\sum_{j=1}^{n}z_{j}'=k'$ (if $\sum_{j=1}^{n}z_{j}'<k'$, then this point is also a vertex of $\textup{conv}(T)$, therefore integral and belonging to $T^{k'}$ - a contradiction).

Since $(x',z')$ is not a vertex of $\textup{conv}(T)$, there exists $(w,r)$ such that the vertex $(x',z') $ is the intersection of the face defined by $\{(x,z) \in \rr^{p}_{+} \times [0,1]^n | \sum_{j=1}^{n} z_j = k' \}$ and an edge of $\textup{conv}(T)$ defined as:

\begin{eqnarray}\label{eqn:theedge}
\{(x,z) \in \textup{conv}(T) \,|\, w^{\top}x + r^{\top}z = \gamma\},
\end{eqnarray}
where $\gamma = \mn \{w^{\top}x + r^{\top}z\,|\, (x,z) \in \textup{conv}(T)\} = w^{\top}x'+ r^{\top}z'$. Let $(x^1,z^1)$ and $(x^2,z^2)$ be two feasible points of $T$ that belong to the edge (\ref{eqn:theedge}) such that $(x',z')$ is a convex combination of $(x^1, z^1)$ and $(x^2, z^2)$. Note that $\gamma = w^{\top}x'+ r^{\top}z' = w^{\top}x^1+ r^{\top}z^1 = w^{\top}x^2+ r^{\top}z^2$. Hence, $(x^1,z^1)$ and $(x^2,z^2)$ are two optimal solutions corresponding to the objective function $(w,r)$. Furthermore, due to our selection of $\gamma$, $\sum_{j \in V_2} z_{j}^{1} < k' < \sum_{j \in V_2} z_{j}^{2} $. The edge property ensures that there exists an integral optimal solution $(x^3,z^3)$ with $k^3=\sum_{j \in V_2} z_{j}^{3}=k'$ such that $\sum_{j \in V_2} z_{j}^{1} < k^3 < \sum_{j \in V_2} z_{j}^{2}$. However, this implies that $(x^3,z^3)$ belongs to the edge defined by (\ref{eqn:theedge}). Thus, $(x^3,z^3)$ must be a convex combination of $(x^1, z^1)$ and $(x^2, z^2)$ or equivalently, we must have $(x^3,z^3)=(x',z')$ with $z'$ integral, a contradiction.
\end{proof}


Next we will verify that the edge property holds for the polytope corresponding to the instances of simple TPMC. If $s_i > 1$ for some $i \in V_1$, we can construct a new instance of simple TPMC where we replace the node with $s_i$ identical nodes each with a capacity of $1$. Note that this is a polynomial construction, because the supply, $s_i$, is at most $2|V_2|$ for any $i\in V_1$. It is straightforward to show that the edge property holds for the polytope corresponding to the first instance of simple TPMC if and only if the edge property holds for the polytope corresponding to new instance of simple TPMC. Therefore, in order to prove Theorem \ref{conjectureintproperty} it is sufficient to verify the following result.

\begin{proposition}\label{keyprop}
The edge property holds for simple TPMC instances with $s_i = 1$ for all $i \in V_1$.
\end{proposition}
The rest of this note is a proof of Proposition \ref{keyprop}.


\begin{myclaim}\label{claim0}
Suppose that $(x^{1}, z^1)$ and $(x^2, z^2)$ are optimal solutions to $\mn \{w^{\top}x + r^{\top}z\,|\, (x,z) \in X\}$. Let $k^1= \sum_{j \in V_2}z^1_j$, $k^2= \sum_{j \in V_2}z^2_j$ and $k^1 \leq k^2 - 2$. If there exists two feasible solutions of $X$, namely $(x^3, z^3)$ and $(x^4, z^4)$, such that
\begin{enumerate}
\item $\sum_{j \in V_2}z^3_j = k^1 + 1$ and $\sum_{j \in V_2}z^4_j = k^2 - 1$ and \label{claim0-1}
\item The objective function value of $(x^3, z^3)$ is $\rho - \delta$ and that of $(x^4, z^4)$ is $\rho + \delta$, where $\rho$ is the objective function value of the solution $(x^1, z^1)$ and $\delta \in \mathbb{R}$,
\end{enumerate}
\end{myclaim}
then the edge property holds. \\
\begin{proof} Since $\rho$ is the optimal objective function value, we obtain that $\delta = 0$ since otherwise the objective function value of either $(x^3, z^3)$  or $(x^4, z^4)$  is better than that of $(x^1, z^1)$. Therefore $(x^3, z^3)$ is an optimal solution with  $ k^1 < \sum_{j \in V_2} z^3_j < k^2$.
\end{proof}

In what follows we assume that if $(x^1, z^1)$ is an optimal solution to $\mn \{w^{\top}x + r^{\top}z\,|\, (x,z) \in X\}$, then $x^1$ is integral. Otherwise, we can solve a simple transportation problem with the set of demand nodes $j$ such that $z^1_j = 0$. Since all data are integral, there exists an optimal solution with integral flows. Therefore, we may assume that $x^1$ is integral.

Given an integral point $(\tilde{x}, \tilde{z})$ of $X$, let $S(\tilde{z}) := \{j \in V_2\,|\, \tilde z_j = 0\}$ be the set of nodes in $V_2$ whose demands are met. For  $j \in S(\tilde{z})$, let $I_j(\tilde{x}, \tilde{z}) = \{ i \in V_1\,|\, \tilde{x}_{ij} > 0 \} = \{ i \in V_1\,|\, \tilde{x}_{ij} = 1 \}$ be the set of suppliers that sends one unit to $j$.

Given the optimal solutions $(x^1, z^1)$ and $(x^2, z^2)$, let $F := \left( S(z^1) \setminus S(z^2) \right) \cup \left( S(z^2) \setminus S(z^1) \right)$, $P:= S(z^1) \cap S(z^2)$ and $R:= V_2 \setminus (F \cup P)$. For $j \in F$, observe that only the set $I_j(x^1, z^1)$ or the set $I_j(x^2, z^2)$ is defined. So for $j \in F$, we define $I_j$ as:

\begin{eqnarray}
I_j := \left\{\begin{array}{rl} I_j(x^1, z^1) & \textup{if } j \in S(z^1) \setminus S(z^2) \\
I_j(x^2, z^2) &  \textup{if } j \in S(z^2) \setminus S(z^1). \\
\end{array}\right.
\end{eqnarray}

As a first step towards constructing $(x^3, z^3)$ and $(x^4, z^4)$ required in Claim \ref{claim0}, we construct a bipartite (conflict) graph $G^*(U_1 \cup U_2, \mathcal{E})$.  The set of nodes is constructed as follows:
\begin{enumerate}
\item If $j \in S(z^1) \setminus S(z^2)$, then $j \in U_1$ and $j$ is called a \emph{full node}. Let $W_1 = S(z^1) \setminus S(z^2)$ be the set of full nodes of $U_1$.
\item Similarly, if $j \in S(z^2) \setminus S(z^1)$, then $j \in U_2$ and $j$ is called a \emph{full node}. Let $W_2 = S(z^2) \setminus S(z^1)$ be the set of full nodes of $U_2$.
\item If $j \in S(z^1) \cap S(z^2)$ and $d_j=2$, then we place two copies of node $j$ in $U_1$  (call these $j_1$ and $j_2$) and two copies of $j$ in $U_2$ (call these $j_3$ and $j_4$). These nodes are called \emph{partial nodes} of $j$. Each partial node of $j$ is distinct: If $I_j(x^1, z^1) = \{t_1, t_2\}$, then associate (WLOG) $t_1$ with $j_1$ and $t_2$ with $j_2$, that is define $I_{j_1}: = \{t_1\}$ and $I_{j_2}:	= \{t_2\}$. Similarly if $I_j(x^2, z^2) = \{t_3, t_4\}$, then associate (WLOG) $t_3$ with $j_3$ and $t_4$ with $j_4$, that is define $I_{j_3} := \{t_3\}$ and $I_{j_4}:= \{t_4\}$. If {$j \in S(z^1) \cap S(z^2)$} and  $d_j=1$, then we place one copy of node $j$ in $U_1$ (call this $j_1$) and one copy of $j$ in $U_2$ (call this $j_3$). Similar to the $d_j=2$ case these nodes are called \emph{partial nodes} of $j$. If $I_{j}(x^1, z^1) = \{t_1\}$ and $I_{j}(x^2, z^2) = \{t_3\}$, then set $I_{j_1} = \{{t}_1\}$ and $I_{j_3} = \{t_3\}$.  Let $P = P^1 \cup P^2$, where $P^1 = \{ j  \in P: d_j = 1\}$ and $P^2 = \{ j  \in P: d_j = 2\}$.
\end{enumerate}
Thus, $U_1 = W_1\cup \left( \bigcup_{j \in P^2} \{j_1, j_2\} \right) \cup \left(\bigcup_{j \in P^1} \{j_1\}\right)$ and for each element $a \in U_1$ the set $I_a$ is well-defined and non-empty. Similarly, $U_2 = W_2 \cup \left(\bigcup_{j \in P^2} \{j_3, j_4\}\right) \cup \left(\bigcup_{j \in P^1} \{j_3\}\right)$ and for each element $b \in U_2$ the set $I_b$ is well-defined and non-empty. Now we construct the edges $\mathcal{E}$ as follows: For all $a \in U_1$ and $b \in U_2$,  there is an edge $(a,b) \in \mathcal{E}$ if and only if $a$ and $b$ have at least one common supplier, i.e.,
\begin{eqnarray}\label{sanegraph}
I_a \cap I_b \neq \emptyset \textup{ iff } (a, b) \in \mathcal{E}.
\end{eqnarray}
Let $G'(V', E')$ be a subgraph of $G^*(U_1 \cup U_2, \mathcal{E})$. Since the elements in $V' \cap (W_1 \cup W_2)$ correspond to unique elements in $V_2$, whenever required we will (with slight abuse of notation) treat $V' \cap (W_1 \cup W_2) \subseteq V_2$.

\begin{myclaim}\label{claim1}
Let $G'(V', E')$ be a subgraph of $G^*(U_1 \cup U_2, \mathcal{E})$ satisfying the following properties:
\begin{enumerate}
\item There are no edges in $G^*$ between the nodes in $V'$ and the nodes in $\left(U_1 \cup U_2 \right)\setminus V'$. \label{property1}
\item For each $j \in P^1$, $| V'  \cap \{j_1\}| = | V'  \cap \{j_3\}| $ and for each $j \in P^2$, $| V'  \cap \{j_1, j_2\}| = | V'  \cap \{j_3, j_4\}|$. \label{property2}
\item $|W_1 \cap V'| =  |W_2 \cap V'| + 1$. \label{property3}
\end{enumerate}
\end{myclaim}
Now construct
\begin{eqnarray}\label{swapdemand}
z^3_j &=& \left\{\begin{array}{cl} z^1_j & \textup{if } j \in V_2 \setminus (V' \cap F) \\
1  & \textup{if } j \in V' \cap W_1\\
0  & \textup{if } j \in V' \cap W_2.
\end{array}\right.\\
\label{swapdemandx} x^3_{ij} &=& \left\{\begin{array}{cl} 1 & \textup{if } j \in F, z^3_j = 0,  i \in I_j \\
1  & \textup{if } j \in P,  j_1 \in (U_1 \cup U_2) \setminus V', i \in I_{j_1} \\
1  & \textup{if } j \in P,  j_2 \in (U_1 \cup U_2) \setminus V', i \in I_{j_2} \\
1  & \textup{if } j \in P,  j_3 \in V', i \in I_{j_3} \\
1  & \textup{if } j \in P,  j_4 \in V', i \in I_{j_4} \\
0  & \textup{otherwise}.
\end{array}\right.
\end{eqnarray}
and
\begin{eqnarray}\label{swapdemand1}
z^4_j &=& \left\{\begin{array}{cl} z^2_j & \textup{if } j \in V_2 \setminus (V' \cap F)\\
0  & \textup{if } j \in V' \cap W_1\\
1  & \textup{if } j \in V' \cap W_2.
\end{array}\right.\\
\label{swapdemand1x} x^4_{ij} &=& \left\{\begin{array}{cl} 1 & \textup{if } j \in F, z^4_j = 0, i \in I_j\\
1  & \textup{if } j \in P,  j_3 \in (U_1 \cup U_2) \setminus V', i \in I_{j_3} \\
1  & \textup{if } j \in P,  j_4 \in (U_1 \cup U_2) \setminus V', i \in I_{j_4} \\
1  & \textup{if } j \in P,  j_1 \in V', i \in I_{j_1}\\
1  & \textup{if } j \in P,  j_2 \in V', i \in I_{j_2}\\
0  & \textup{otherwise}.
\end{array}\right.
\end{eqnarray}
Then $(x^3, z^3)$ and $(x^4, z^4)$ are feasible solutions of $X$ that satisfy the requirements of Claim \ref{claim0}.\\
\begin{proof}
\begin{enumerate}
\item We verify that $(x^3, z^3)$ is a valid solution to $X$. A similar proof can be given for the validity of $(x^4, z^4)$. Clearly $x^3$ and $z^3$ satisfy the variable restrictions. We verify that the constraint $\sum_{i : (i,j) \in E}x^3_{ij} + d_j z_j = d_j$ is satisfied for all $j \in V_2$. If $j \in R$, then $z^3_j = z^1_j = 1$ and $x^3_{ij} = 0$ for all $(i,j) \in E$; therefore the constraint is satisfied. If $j \in F$, then using the first and last entry in (\ref{swapdemandx}), we have  $\sum_{i : (i,j) \in E}x^3_{ij} + d_jz^3_j = d_j$. If $j \in P$, then $j \in V_2\setminus (V' \cap F)$. Therefore $z^3_j = z^1_j = 0$. Now it is straightforward to verify that $\sum_{i: (i,j) \in E}x^3_{ij} = 2 = d_j$ for each $j \in P^2$ since $| V'  \cap \{j_1, j_2\}| = | V'  \cap \{j_3, j_4\}| $ and by the use of the second, third, fourth and fifth entries in (\ref{swapdemandx}). For $j \in P^1$ we have $\sum_{i: (i,j) \in E}x^3_{ij} = 1 = d_j$ since $| V'  \cap \{j_1\}| = | V'  \cap \{j_3\}| $ and by the use of the second and fourth entries in (\ref{swapdemandx}).

Now we verify that the constraint $\sum_{j:(i,j) \in E} x_{ij} \leq 1$ is satisfied for all $i \in V_1$. Given $i \in V_1$, assume for contradiction that $x^3_{i g} = x^3_{i h} = 1$ for some $g, h \in V_2$ and $g \neq h$. By construction of $(x^3, z^3)$, $x^3_{ij} = 0$ for all $j \in R$. Thus, $g, h \notin R$. Moreover since $\sum_{i : (i,j) \in E}x^3_{ij} + d_j z_j = d_j$ is satisfied for all $j \in V_2$, we have  $z^3_g = z^3_h = 0$.
Now, there are three cases to consider:
\begin{enumerate}
\item $g, h \in F$. By construction of $x^3$ we have  $i \in I_{g}  \cap I_{h}$. Now if $g \notin V'$ and $h \notin V'$, then by construction of $z^3$  (first entry in (\ref{swapdemand})) we have  $z^1_g = z^3_g =  0 = z^3_h = z^1_h$ and thus $g, h \in S(z^1)$. Therefore by the validity of $(x^1, z^1)$ we have  $I_g \cap I_h = \emptyset$. This contradicts $i \in I_g \cap I_h$. Now consider the case where $g \in V'$ and $h \in V'$. Since $i \in I_g \cap I_h$ by (\ref{sanegraph}) there is an edge between $g$ and $h$ in $G^*(U_1\cup U_2, \mathcal{E})$. Thus we may assume without loss of generality that $g \in V' \cap W_1$ and $h \in V' \cap W_2$. However, this implies that  $z^3_g =1$, a contradiction. Now, without loss of generality, assume that $g \in V'$ and $h \notin V'$. Since $i \in I_g \cap I_h$ by (\ref{sanegraph}) there is an edge between $g$ and $h$ in $G^*(U_1\cup U_2, \mathcal{E})$. On the other hand, by assumption there is no edge between nodes in $V'$ and those not in $V'$, which is the required contradiction.
\item $g \in F$ and $h \in P$. We assume that $g \in W_1$ (the proof when $g \in W_2$ is similar). If $g \in V'$, then $z^3_g = 1$, a contradiction. Therefore, we have  $g \notin V'$. Thus $z^1_g = z^3_g = 0$. Therefore by validity of $(x^1, z^1)$ we have $i \notin I_{h}(x^1, z^1)$ or equivalently $i \in I_{h}(x^2, z^2)$. Without loss of generality, we may assume that $i \in I_{h_3}$. Note that $h_3$ belongs to $V'$ (by the construction of $x^3$ and the fact that $x^3_{ih} = 1$ and $i \in I_{h_3}$). Since $i \in I_g$,  there exists an edge between $g$ and $h_3$. However, since $g \notin V'$ and $h_3 \in V'$, we get a contradiction to the fact that there are no edges between the nodes in $V'$ and the nodes in $(U_1 \cup U_2) \setminus V'$.
\item $g, h \in P$. In this case, we may assume without loss of generality that $i \in I_g(x^1, z^1)$ and $i \in I_h(x^2, z^2)$. Therefore without loss of generality, we may assume that $i \in I_{g_1}$ and $i \in I_{h_3}$. Since $x^3_{ig} = x^3_{ih} = 1$, we have  $g_1 \notin V'$ and $h_3 \in V'$. By assumption on $G'$, this implies that there is no edge between $g_1$ and $h_3$. On the other hand, since $i \in I_{g_1} \cap I_{h_3}$ by (\ref{sanegraph}) we have an edge $(g_1, h_3) \in \mathcal{E}$, a contradiction.
\end{enumerate}
\item Next we verify that the objective function value of $(x^3, z^3)$ is $\rho - \delta$ and that of $(x^4, z^4)$ is $\rho + \delta$ where $\rho$ is the objective function value of the solution $(x^1, z^1)$ and $\delta \in \mathbb{R}$. This result is verified by showing that $(x^3,z^3)$ and $(x^4, z^4)$ are obtained by `symmetrically' updating demands from $(x^1, z^1)$ and $(x^2, z^2)$ respectively. In particular, we examine each demand node and examine the cost of either satisfying it or not satisfying it in each solution. We consider the different cases next:
\begin{enumerate}
\item $j \in R$. Then  $z^4_j = z^3_j = z^1_j = z^2_j = 1$.
\item $j \in V' \cap W_1$. Then $z^1_j = 0$ and $z^3_j=1$. On the other hand  $z^2_j=1$ and $z^4_j=0$. Notice that in each solution where $d_j$ is satisfied, this is done by using the same set of input nodes (and thus using the same arcs).  Therefore the difference in objective function value between $(x^1, z^1)$ and $(x^3, z^3)$ due to demand node $j$ is $-\sum_{i \in I_j}w_{ij} + r_j$ and the difference in objective function value between the solutions $(x^2, z^2)$ and $(x^4, z^4)$ due to demand node $j$ is $\sum_{i \in I_j}w_{ij} - r_j$.
\item $j \in V' \cap W_2$. Similar to the above case the difference in objective function value between $(x^1, z^1)$ and $(x^3, z^3)$ due to demand node $j$ is $\sum_{i \in I_j}w_{ij} - r_j$ and the difference in objective function value between $(x^2, z^2)$ and $(x^4, z^4)$ due to demand node $j$ is $-\sum_{i \in I_j}w_{ij} + r_j$.
\item $j \in F\setminus V' $, then  $z^1_j = z^3_j$ and $x^1_{ij} = x^3_{ij}$ for all $(i, j)\in E$. Similarly $z^2_j = z^4_j$ and $x^2_{ij} = x^4_{ij}$ for all $(i, j)\in E$.
\item $j \in P^2$ such that $j_1, j_2 \in (U_1 \cup U_2) \setminus V'$  and $j_3, j_4 \in (U_1 \cup U_2) \setminus V'$. Then the demand $d_j$ is satisfied by the nodes in $I_{j}(x^1, z^1)$ in $(x^1, z^1)$ and $(x^3, z^3)$. Therefore there is no difference in objective function value between $(x^1, z^1)$ and $(x^3, z^3)$ with respect to demand node $j$. Similarly, the demand $d_j$ is satisfied by the nodes in $I_{j}(x^2, z^2)$ in $(x^2, z^2)$ and $(x^4, z^4)$ and there is no difference in objective function value between $(x^2, z^2)$ and $(x^4, z^4)$ with respect to demand node $j$. We can make a similar argument for $j \in P^1$ such that $j_1 \in (U_1 \cup U_2) \setminus V'$  and $j_3\in (U_1 \cup U_2) \setminus V'$.
\item $j \in P^2$ such that $j_1 \in V'$, $j_2 \in (U_1 \cup U_2) \setminus V'$, $j_3 \in (U_1 \cup U_2) \setminus V'$, $ j_4 \in V'$ without loss of generality. Then the demand $d_j$ is satisfied by the nodes in $(I_{j_1} \cup I_{j_2})$ in $(x^1, z^1)$ and  by nodes $(I_{j_2} \cup I_{j_4})$ in $(x^3, z^3)$. Therefore the difference in objective function value between $(x^1, z^1)$ and $(x^3, z^3)$ with respect to demand node $d_j$ is $\sum_{i \in I_{j_1}}w_{ij} - \sum_{i \in I_{j_4}}w_{ij}$. The demand $d_j$ is satisfied by the nodes in $(I_{j_3} \cup I_{j_4})$ in $(x^2, z^2)$ and  by the nodes in  $(I_{j_1} \cup I_{j_3})$ in $(x^4, z^4)$. Therefore the difference in objective function value between $(x^2, z^2)$ and $(x^4, z^4)$ with respect to demand node $j$ is $\sum_{i \in I_{j_4}}w_{ij} - \sum_{i \in I_{j_1}}w_{ij}$. We can make a similar argument for the cases:  $j_1 \in (U_1 \cup U_2) \setminus V'$, $j_2 \in V'$, $j_3 \in V'$, $ j_4 \in (U_1 \cup U_2) \setminus V'$; $j_1 \in V'$, $j_2 \in (U_1 \cup U_2) \setminus V'$, $j_3 \in V'$, $ j_4 \in (U_1 \cup U_2) \setminus V'$ and $j_1 \in (U_1 \cup U_2) \setminus V'$, $j_2 \in V'$, $j_3 \in (U_1 \cup U_2) \setminus V'$, $ j_4 \in V'$.
\item $j \in P^2$ such that $j_1 \in  V'$, $j_2 \in V'$, $j_3 \in V'$, $ j_4 \in V'$. Then the demand $d_j$ is satisfied by the nodes in $(I_{j_1} \cup I_{j_2})$ in $(x^1, z^1)$ and  by the nodes in  $(I_{j_3} \cup I_{j_4})$ in $(x^3, z^3)$. Therefore, the difference in the objective function value between $(x^1, z^1)$ and $(x^3, z^3)$ with respect to satisfying demand $d_j$ is $\sum_{i \in (I_{j_1} \cup I_{j_2})}(w_{ij} + w_{ij})  - \sum_{i \in (I_{j_3} \cup I_{j_4})}(w_{ij} + w_{ij}) $. The demand $d_j$ is satisfied by the nodes in $(I_{j_3} \cup I_{j_4})$ in $(x^2, z^2)$ and  by the nodes in  $(I_{j_1} \cup I_{j_2})$ in $(x^4, z^4)$. Therefore, the difference in the objective function value between $(x^2, z^2)$ and $(x^4, z^4)$ with regards to satisfying demand $d_j$ is $-\sum_{i \in (I_{j_1} \cup I_{j_2})}(w_{ij} + w_{ij})  + \sum_{i \in (I_{j_3} \cup I_{j_4})}(w_{ij} + w_{ij})$. For $j \in P^1$, we can similarly consider $j_1$ and $j_3$ with $j_1 \in V'$, $j_3 \in V'$.
\end{enumerate}
Therefore, the objective function value of $(x^3, z^3)$ is $\rho - \delta$ and that of $(x^4, z^4)$ is $\rho + \delta$ where $\rho$ is the objective function value of the solution $(x^1, z^1)$ and $(x^2, z^2)$ and $\delta \in \mathbb{R}$.
\item Finally we verify that $\sum_{j \in V_2}z^3_j = k^1 +1$ and $\sum_{j \in V_2}z^4_j = k^2 - 1$. We prove this for $(x^3, z^3)$. The proof is similar for the case of $(x^4, z^4)$. Observe that if $j \in R$, then $z^1_j = z^3_j = 1$. If $j \in P$, then $z^1_j = z^3_j = 0$. If $j \in F\setminus V'$, then $z^1_j = z^3_j$. If $j \in W_1 \cap V'$, then $z^1_j = 0$ and $z^3_j = 1$ and if $j \in W_2 \cap V'$, then $z^1_j = 1$ and $z^3_j = 0$. Thus $\sum_{j \in V_2}z^1_j - \sum_{j \in V_2}z^3_j = |V' \cap W_2| - |V' \cap W_1| =  -1$, where the last equality is by assumption (\ref{property3}) of $G'$. Thus, $\sum_{j \in V_2}z^3_j = k^1 +1$.
\end{enumerate}
\end{proof}
Now the proof of Theorem \ref{conjectureintproperty}  is complete by showing that a subgraph $G'(V',E')$ of $G^*(U_1 \cup U_2, \mathcal{E})$ always exists that satisfies the conditions of Claim \ref{claim1}. In order to prove this, we verify a few results.

\begin{myclaim} \label{claim2}
Connected components of $G^*$  are paths or cycles of even length and all the cycles involve only full nodes.
\end{myclaim}
\begin{proof} This is evident from the fact that $G^*$ is bipartite and degree of $a \in (U_1 \cup U_2)$ is bounded from above by $|I_a|$.\hfill  
\end{proof}
We associate a value $v_j$ to each node $j \in U_1 \cup U_2$. In particular:
\begin{enumerate}
\item If $j \in W_1$, then $v_j = 1$.
\item If $j \in U_1$ and $j$ is a partial node, then $v_j = \frac{1}{2}$.
\item If $j \in U_2$ and $j$ is a partial node, then $v_j = -\frac{1}{2}$.
\item If $j \in W_2$, then $v_j = -1$.
\end{enumerate}
For a { subgraph} $\tilde{G}(\tilde{V}, \tilde{E})$ { of $G^{*}$} we call $v(\tilde{V}) = \sum_{j \in \tilde{V}}v_j$ the \emph{value of the path}.

\begin{myclaim}\label{claim3}
$v(U_1 \cup U_2) = k^2 - k^1 \geq 2$.
\end{myclaim}
\begin{proof} $\sum_{j \in U_1 \cup U_2}v_j = \sum_{j \in W_1}v_j + \sum_{j \in P^2} (v_{j_1} + v_{j_2}) + \sum_{j \in P^1} v_{j_1} + \sum_{j \in W_2}v_j + \sum_{j \in P^2} (v_{j_3} + v_{j_4}) + \sum_{j \in P^1} v_{j_3} = |S(z^1) \setminus S(z^2)|  - |S(z^2) \setminus S(z^1)| = |S(z^1)| - |S(z^2)| = k^2 - k^1$.
\end{proof}
\begin{myclaim}\label{claim4}
If $\tilde{G}(\tilde{V},\tilde{E})$ is a cyclic subgraph of $G^*(U_1 \cup U_2, \mathcal{E})$, then $v(\tilde{V}) = 0$.
\end{myclaim}
\begin{proof} By Claim \ref{claim2}, a cycle has only full nodes. Moreover, since a cycle is of even length, it contains equal number of nodes from $W_1$ and $W_2$.
\end{proof}

In the rest of this note, when we refer to a path in $G^*$, we refer to a connected component of $G^*$ which is a path (that is any node not belonging to the path does not share an edge with any node of the path).

Note that a partial node must be a leaf node in a path. Using this observation and by some simple case analysis the following three claims can be verified.
\begin{myclaim}\label{claim5}
If $\tilde{G}(\tilde{V}, \tilde{E})$ is a path containing exactly one partial node, then $v(\tilde{V}) \in \{-\frac{1}{2}, \frac{1}{2}\}$.
\end{myclaim}
\begin{myclaim}\label{claim6}
If $\tilde{G}(\tilde{V}, \tilde{E})$ is a path containing two partial nodes, then $v(\tilde{V}) = 0$.
\end{myclaim}
\begin{myclaim}\label{claim7}
If $\tilde{G}(\tilde{V}, \tilde{E})$ is a path containing only full nodes, then $v(\tilde{V}) \in \{-1, 0, 1\}$.
\end{myclaim}
\noindent For the subgraph {$\tilde{G}(\tilde{V}, \tilde{E})$, consider a $k \in \tilde{V} \setminus F$} such that  $k = j_t$ where $t \in \{1, 2, 3, 4\}$ and $j \in P^2$.  Suppose $k = j_1$ or $j_2$, then we say that a path $\check{G}(\check{V}, \check{E})$ is a \emph{mirror path} for $j$, if $\check{V}$ contains either $j_3$ or $j_4$. Moreover we call one of $j_3$ or $j_4$ (whichever belongs to $\check{V}$ or arbitrarily select one of these if both belong to $\check{V}$) as the \emph{mirror node}. Similarly if $k = j_3$ or $j_4$, then we say that a path $\check{G}(\check{V}, \check{E})$ is a \emph{mirror path} for $j$, if $\check{V}$ contains either $j_1$ or $j_2$. \emph{Mirror node} is similarly defined in this case. For $j \in P^1$ we consider $k=j_1$ and $k=j_3$. Suppose $k=j_1$, then we say that a path $\check{G}(\check{V}, \check{E})$ is a \emph{mirror path} for $j$, if $\check{V}$ contains $j_3$ and we call $j_3$ the \emph{mirror node}. Similarly if $k=j_3$, then we say that a path $\check{G}(\check{V}, \check{E})$ is a \emph{mirror path} for $j$, if $\check{V}$ contains $j_1$ and we call $j_1$ the \emph{mirror node}.

Algorithm \ref{algotable} constructs $G'(V', E')$ that satisfies all the properties of Claim \ref{claim1}. We next verify that Algorithm \ref{algotable} is well-defined, that is all the steps can be carried out. Moreover, we show that the algorithm  generates a subgraph $G'(V', E')$ that satisfies the conditions of Claim \ref{claim1}.

\begin{algorithm}                      
\caption{Construction of $G'(V', E')$}           
\label{algotable}                      
\noindent \textbf{Input:} $G^*(U_1 \cup U_2, \mathcal{E})$.\\
\noindent  \textbf{Output:} $G'(V', E')$ that satisfies all conditions of Claim \ref{claim1}.\\
\begin{enumerate}
      \item \label{cleansol}If there exists a path $\tilde{G}(\tilde{V}, \tilde{E})$ in $G^*(U_1 \cup U_2, \mathcal{E})$ containing only full nodes with $v(\tilde{V}) = 1$, then set $G':= \tilde{G}$. STOP.

      \item Tag all paths in $G^*(U_1 \cup U_2, \mathcal{E})$ as `unmarked.'

      \item \label{atzero}Select a path $\tilde{G}(\tilde{V}, \tilde{E})$ from the set of `unmarked' paths containing a partial node  such that $v(\tilde{V}) = \frac{1}{2}$. Tag this path as `marked.' Note that by Claim \ref{claim5} and Claim \ref{claim6},  $\tilde{V}$ contains a unique partial node $j^{*}$.	

      \item \label{athalf} Select a path from the list of `unmarked' paths, such that it is a mirror path for $j^{*}$. Tag this path as `marked.'

      \item \label{decision} There are three cases:
	       \begin{enumerate}
	       \item The mirror path tagged as `marked' in (\ref{athalf}) contains a unique partial node and its value is $\frac{1}{2}$. \newline GO TO Step \ref{laststep}.
	       \item The mirror path tagged as `marked' in (\ref{athalf}) contains a unique partial node and its value is $-\frac{1}{2}$. \newline GO TO Step \ref{atzero}.
	       \item The mirror path tagged as `marked' in (\ref{athalf}) contains two partial nodes (then its value is $0$): \newline One of the partial nodes {corresponds to} the mirror node. Set $j^{*}$ to be the other partial node. GO TO Step \ref{athalf}.
	       \end{enumerate}

       \item \label{laststep} Set $G'(V', E')$ to be disjoint union of the paths tagged as `marked.' STOP.
\end{enumerate}
\end{algorithm}

\begin{myclaim}\label{claim8}
Algorithm \ref{algotable} is well-defined.
\begin{enumerate}
\item At the beginning of Step (\ref{atzero}), the total value of all marked paths is $0$.
\item Let $\hat{V} := \bigcup_{ \tilde{G}(\tilde{V}, \tilde{E}) \textup{ is marked before Step (\ref{atzero})} } \tilde{V} $. Then $|\hat{V} \cap \{j_1, j_2 \}| = |\hat{V} \cap \{j_3, j_4 \}|$ for all $j \in P^2$ and $|\hat{V} \cap \{j_1\}| = |\hat{V} \cap \{j_3\}|$ for all $j \in P^1$.
\item Step (\ref{atzero}) is well-defined, that is as long as the algorithm does not terminate, Step (\ref{atzero}) can be carried out.
\item At the end of Step (\ref{atzero}), the total value of all marked paths is $\frac{1}{2}$.
\item Step (\ref{athalf}) is well-defined, that is as long as the algorithm does not terminate, Step (\ref{athalf}) can be carried out.
\end{enumerate}

\end{myclaim}
\begin{proof} We prove Claim \ref{claim8} by induction on the iteration number ($n$) of the algorithm visiting Step (\ref{decision}). When $n = 0$:
\begin{enumerate}
\item At the beginning of Step (\ref{atzero}) there are no `marked' paths and therefore the total value of all marked paths is $0$.
\item $\hat{V} = \emptyset$.
\item By Step (\ref{cleansol}), we know that there exists no path containing only full nodes with $v(\tilde{V}) = 1$. Moreover by Claim \ref{claim3} we have  $v(U_1 \cup U_2) \geq 2$. Since by Claim \ref{claim4} all cycles have a value of $0$, there must exist at least one path with partial nodes with positive value. Since this is only possible (Claim \ref{claim5},    Claim \ref{claim6} and Claim \ref{claim7}) if there exists exactly one partial node in the path, we see that Step (\ref{atzero}) is well-defined.
\item At Step (\ref{atzero}) one path is marked which has a value of half.
\item Since one path is tagged as marked in Step (\ref{atzero}), it contains exactly one partial node, $j^*\in P$. Suppose that $j^* \in P^2$ and  $j^* = {j^{*}_i}$ for some $i \in \{1, \dots, 4\}$. Then there exists paths (at least two) which contain the other three partial nodes corresponding to $j^*$. If $j^*\in P^1$ then there exists one path which contains the other partial node. Therefore this step is well-defined.
\end{enumerate}

Now for any $n \in \mathbb{Z}_{+}$, assuming by the induction hypothesis that the result is true for $n' = 0, \dots, n - 1$:
\begin{enumerate}
\item Step (\ref{atzero}) is arrived at via Step (\ref{decision}b). Let $n' < n$ be the last iteration when Step (\ref{atzero}) is invoked. By the induction hypothesis, the total value of all the marked paths at the end of Step (\ref{atzero}) in iteration $n'$ is $\frac{1}{2}$.
From iterations $n' + 1, \dots, n -1$, the algorithm alternates between Step (\ref{athalf}) and Step (\ref{decision}c). The total value of all the marked paths here is $0$. Finally, the value of the  last path tagged as marked in Step (\ref{athalf}) is $-\frac{1}{2}$ (since the algorithm invokes Step (\ref{decision}b)). Hence, the total value of all the marked paths is $0$ at the beginning of Step (\ref{atzero}) in iteration $n$.
\item \label{pickboth} Let $n' < n$ be the last iteration when Step (\ref{atzero}) is invoked. By the induction hypothesis $|\hat{V} \cap \{j_1, j_2 \}| = |\hat{V} \cap \{j_3, j_4 \}|$ for all $j \in P^2$ and $|\hat{V} \cap \{j_1\}| = |\hat{V} \cap \{j_3\}|$ for all $j \in P^1$ where \newline $\hat{V} := \bigcup_{ \tilde{G}(\tilde{V}, \tilde{E}) \textup{ is marked before Step (\ref{atzero}) \textup{ iteration } n' } } \tilde{V} $. From iterations $n' + 1, \dots, n -1$, the algorithm alternates between Step (\ref{athalf}) and Step (\ref{decision}c). Since in iteration $n -1$ at Step (\ref{athalf}), we add one path that contains only the mirror node to $j^{*}$ (the unique partial node from the previous iteration), we arrive at this result.

\item Proof is the same as that in the case where $n = 0$.
\item The total value of paths at the end of Step (\ref{atzero}) = value of marked path  + total value of previously marked path $= \frac{1}{2} + 0 $.
\item Step (\ref{athalf}) is invoked after either Step (\ref{atzero}) or Step (\ref{decision}c). In case we arrive via Step (\ref{atzero}), by the induction hypothesis, $|\hat{V} \cap \{j_1, j_2 \}| = |\hat{V} \cap \{j_3, j_4 \}|$ for all $j \in P^2$ and $|\hat{V} \cap \{j_1\}| = |\hat{V} \cap \{j_3\}|$ for all $j \in P^1$ where $\hat{V} := \bigcup_{ \tilde{G}(\tilde{V}, \tilde{E}) \textup{ is marked before Step (\ref{atzero}) \textup{ iteration } n' } } \tilde{V} $. Moreover the path marked in Step (\ref{atzero}) contains exactly a unique partial node $j^{*}$, then there must exist an unmarked path containing a mirror node to $j^{*}$. In case of we arrive via Step (\ref{decision}c), again the proof is essentially the same by observing that at the start of Step (\ref{athalf}), there is a unique partial node $j^{*}$ that is not paired with a mirror partial node.

\end{enumerate}
\end{proof}
\begin{myclaim}\label{claim9}
Algorithm \ref{algotable} terminates in finite time.
\end{myclaim}
\begin{proof} This is true since there are a finite number of edges and at each iteration of the algorithm at least one unmarked path is tagged as marked.
\end{proof}
\begin{myclaim}\label{claim10}
Algorithm \ref{algotable} generates  a subgraph $G'(V', E')$  that satisfies the properties of Claim \ref{claim1}.
\end{myclaim}
\begin{proof} First observe that since the output $G'(V', E')$ of the algorithm is a disjoint union of paths, there exists no edge between $V'$ and $(U_1 \cup U_2)\setminus V'$ in $\mathcal{E}$, so property \ref{property1} is satisfied.

By Claim \ref{claim8}, 2. we have  $|\hat{V} \cap \{j_1, j_2 \}| = |\hat{V} \cap \{j_3, j_4 \}|$ for all $j \in P^2$ and $|\hat{V} \cap \{j_1\}| = |\hat{V} \cap \{j_3\}|$ for all $j \in P^1$ where $$\hat{V} := \bigcup_{ \tilde{G}(\tilde{V}, \tilde{E}) \textup{ is marked before Step (\ref{atzero})} } \tilde{V}.$$
Therefore, it is easily verified that in the last iteration before termination, a path with a unique partial node, which is a mirror node to $j^{*}$, is marked in Step (\ref{athalf}). This is because before termination we arrive at Step (\ref{decision}a) implying that the value of the path marked in Step (\ref{athalf}) is $\frac{1}{2}$. Hence Claim \ref{claim5} and Claim \ref{claim6} imply that there is a unique partial node in this path. Thus, $|V' \cap \{j_1, j_2 \}| = |V' \cap \{j_3, j_4 \}|$ for all $j \in P^2$ and $|V' \cap \{j_1\}| = |V' \cap \{j_3\}|$ for all $j \in P^1$, so property \ref{property2} is satisfied.

Finally, since $v(V') = 1$ and $|V' \cap \{j_1, j_2 \}| = |V' \cap \{j_3, j_4 \}|$ for all $j \in P^2$ and $|V' \cap \{j_1\}| = |V' \cap \{j_3\}|$ for all $j \in P^1$ we have $$\sum_{j \in V' \cap W_1}v_j + \sum_{j \in V' \cap W_2}v_j = 1.$$ As a result, $|V' \cap W_1| = |V' \cap W_2| + 1$, so property \ref{property3} is satisfied.
\end{proof}

We showed that the set of solutions to TPMC satisfies the edge property. Theorem \ref{conjectureintproperty} then follows from Proposition \ref{edgproperty-convexhull}.

 \section*{Acknowledgements.} Pelin Damc\i-Kurt and Simge K\"u\c{c}\"ukyavuz are supported, in part,  by NSF-CMMI grant 1055668.  Santanu S. Dey gratefully acknowledges the support of the AIR Force Office of Scientific Research grant FA9550-12-1-0154.

\bibliographystyle{plain}
\bibliography{reference}

\section*{Appendix A}

\begin{proposition} Theorem \ref{conjectureintproperty} implies the result that the matching polytope intersected with a cardinality constraint is integral.
\end{proposition}
\begin{proof} We use the notation introduced in Observation \ref{mathobs}. Let $e$ be the all ones vector in $\mathbb{R}^m$. Therefore we need to prove that the matching polytope intersected with a cardinality constraint is integral, i.e.,
\begin{eqnarray}
\textup{conv}\left( \textup{proj}_y(H) \cap \left\{y \,|\, e^{\top}y \leq k \right\} \right) = \textup{conv}\left( \textup{proj}_y(H)\right) \cap \left\{y \,|\, e^{\top}y\leq k \right\},
\end{eqnarray}
using Theorem \ref{conjectureintproperty}.

By Theorem \ref{conjectureintproperty}, we have
\begin{eqnarray}
\label{thmagain} \textup{conv}\left(Q \cap \left\{(x,z) \,|\, e^{\top}z \geq m - k \right\} \right) = \textup{conv}\left( Q\right) \cap \left\{(x,z) \,|\, e^{\top}z\geq m - k \right\}.
\end{eqnarray}
By appending the constraints $y = e -z $ on both sides of (\ref{thmagain}), we obtain
\begin{eqnarray}
\label{thmagain1} \textup{conv}\left(H \cap \left\{(x,z, y) \,|\, e^{\top}z \geq m-k \right\} \right) = \textup{conv}\left(H \right) \cap \left\{(x,z, y) \,|\, e^{\top}z\geq m-k \right\},
\end{eqnarray}
Since $ y = e - z$ in the definition of $H$, (\ref{thmagain1}) is equivalent to
\begin{eqnarray}
\label{thmagain2} \textup{conv}\left(H \cap \left\{(x,z, y) \,|\, e^{\top}y \leq k \right\} \right) &=& \textup{conv}\left(H \right) \cap \left\{(x,z, y) \,|\, e^{\top}y\leq k \right\} \nonumber \\
\label{preproj} \Rightarrow \textup{proj}_y \left( \textup{conv}\left(H \cap \left\{(x,z, y) \,|\, e^{\top}y \leq k \right\} \right)\right) &= & \textup{proj}_y \left(\textup{conv}\left(H \right) \cap \left\{(x,z, y) \,|\, e^{\top}y\leq k \right\} \right).
\end{eqnarray}
Now note that
\begin{eqnarray}
\textup{proj}_y \left( \textup{conv}\left(H \cap \left\{(x,z, y) \,|\, e^{\top}y \leq k \right\} \right)\right) &= & \textup{conv}\left(\textup{proj}_y\left(H \cap \left\{(x,z, y) \,|\, e^{\top}y \leq k \right\} \right)\right) \nonumber \\
\label{projleft}&=& \textup{conv}\left( \textup{proj}_y(H) \cap \left\{y \,|\, e^{\top}y \leq k \right\} \right),
\end{eqnarray}
where the second equality follows from the fact that the  constraint $e^{\top}y \leq k$ is independent of $x$ and $z$.

Also note that
\begin{eqnarray}
\textup{proj}_y \left(\textup{conv}\left(H \right) \cap \left\{(x,z, y) \,|\, e^{\top}y\leq k \right\} \right) &=& \textup{proj}_y \left(\textup{conv}\left(H \right)\right) \cap \textup{proj}_y \left(\left\{(x,z, y) \,|\, e^{\top}y\leq k
\right\} \right) \nonumber\\
\label{projright}&=&  \textup{conv}\left(\textup{proj}_y\left(H \right)\right) \cap \left\{y \,|\, e^{\top}y\leq k \right\},
\end{eqnarray}
where again the first equality follows from the fact that the  constraint $e^{\top}y \leq k$ is independent of $x$ and $z$. Equations (\ref{preproj})-(\ref{projright}) complete the proof.
\end{proof}

\end{document}